\documentclass[a4paper,11pt]{article} 
\title{Average  Characteristic Polynomials of \\
Determinantal Point Processes} 

\author{
\ Adrien Hardy 
\footnote{Institut de Math\'ematiques de Toulouse, Universit\'e de Toulouse, 31062 Toulouse, France.} \hspace{2 dd}\footnote{Department of Mathematics, KU Leuven, Celestijnenlaan 200 B,
3001 Leuven, Belgium. Email address: adrien.hardy@wis.kuleuven.be} }

\usepackage[english]{babel}
\usepackage{anysize}
\marginsize{3.5cm}{3.5cm}{2.5cm}{2.5cm}
\usepackage{fancyhdr} 
\usepackage{amsfonts}
\usepackage{amsmath} 
\usepackage{tikz,graphicx,subfigure,overpic}
\usepackage{color} 
\usepackage{amssymb} 
\usepackage{amsthm}
\usepackage[colorlinks=true, linkcolor=blue, citecolor=blue]{hyperref}
\usepackage[fixlanguage]{babelbib}
\usepackage{color}
\numberwithin{equation}{section}

\def\Tr{\mathop{\mathrm{Tr}}\nolimits}

\newtheorem{theorem}{Theorem}[section]
\newtheorem{lemma}[theorem]{Lemma}
\newtheorem{corollary}[theorem]{Corollary}
\newtheorem{proposition}[theorem]{Proposition}
\theoremstyle{definition} 
\newtheorem{definition}[theorem]{Definition}
\newtheorem{assumption}[theorem]{Assumption}
\newtheorem{Remark}[theorem]{Remark}
\newenvironment{remark}{\begin{Remark}\rm}{\end{Remark}}

\newtheorem{Example}[theorem]{Example}

\newcommand{\eq}{\begin{equation}}
\newcommand{\qe}{\end{equation}}

\newcommand{\tr}{{\rm Tr}}
\newcommand{\R}{\mathbb{R}}
\newcommand{\C}{\mathbb{C}}

\newcommand{\N}{\mathbb{N}}
\newcommand{\E}{\mathbb{E}}

\newcommand{\bs}{\boldsymbol}
\newcommand{\bv}{\mathbf}
\newcommand{\bt}{\textbf}

\newcommand{\2}{{\rm II}}

\newcommand{\Span}{{\rm Span}}
\newcommand{\im}{{\rm Im}}

\newcommand{\di}{{\mathrm d}}

\renewcommand{\leq}{\leqslant}
\renewcommand{\geq}{\geqslant}

\begin{document}

\maketitle 

\begin{abstract}
We investigate the average characteristic  polynomial $\mathbb E\big[\prod_{i=1}^N(z-x_i)\big] $ where the $x_i$'s are  real random variables drawn from a Biorthogonal Ensemble, i.e.  a determinantal point process associated with a bounded finite-rank projection operator. For a   subclass of  Biorthogonal Ensembles, which contains Orthogonal Polynomial Ensembles and (mixed-type) Multiple Orthogonal Polynomial Ensembles, we provide a sufficient condition for its limiting zero  distribution to match with the  limiting distribution of the random variables, almost surely, as $N$ goes to infinity.  Moreover, such a condition  turns out to be  sufficient to strengthen the mean convergence to the almost sure one for the moments of the empirical measure associated to the determinantal point process, a fact of independent interest. As an application, we obtain from Voiculescu's theorems the  limiting zero distribution for multiple Hermite and multiple Laguerre polynomials, expressed in terms of  free convolutions of classical distributions with atomic measures, and then derive explicit algebraic equations for their Cauchy-Stieltjes transform. 
\end{abstract}

\begin{abstract}
On s'int\'eresse au polyn\^ome caract\'eristique moyen $\mathbb E\big[\prod_{i=1}^N(z-x_i)\big]$ associ\'e \`a des variables al\'eatoires r\'eelles  $x_1,\ldots,x_N$ qui forment un Ensemble Biorthogonal, c'est-\`a-dire  un processus ponctuel d\'eterminantal  associ\'e \`a un op\'erateur de projection born\'e et de rang fini. Pour une sous-classe d'Ensembles Biorthogonaux, qui  contient les Ensembles Polyn\^omes Orthogonaux et  les Ensembles Polyn\^omes Orthogonaux Multiples (de type mixte), nous obtenons une condition suffisante  pour que, presque s\^urement, la distribution limite de ses z\'eros coincide  avec la distribution limite des variables al\'eatoires, quand $N$ tend vers l'infini.  De plus, cette  condition  s'av\`ere \^etre \'egalement suffisante pour am\'eliorer la convergence en moyenne en convergence presque s\^ure  pour les moments de la mesure empirique associ\'ee au processus ponctuel d\'eterminantal. En application, on obtient avec des th\'eor\`emes de Voiculescu une description pour les distributions limites des z\'eros des polyn\^omes d'Hermite et de Laguerre multiples, en termes de convolutions libres de lois classiques avec des mesures atomiques, ainsi que des \'equations alg\'ebriques explicites pour leurs transform\'ees de Cauchy-Stieltjes. 
 
\end{abstract}

\section{Introduction and statement of the results}

\subsection{Introduction}
For any $N\geq 1$, let $x_1,\ldots,x_N$ be a collection of real random variables which forms a Biorthogonal Ensemble, that is a determinantal point process associated with a rank $N$ bounded projection operator. This means there exists for each $N$ a Borel measure $\mu_N$ on $\R$ and  two families $(P_{k,N})_{k=0}^{N-1}$ and $(Q_{k,N})_{k=0}^{N-1}$ of $L^2(\mu_N)$-functions which are biorthogonal, namely which satisfies 
\eq
\label{biortho}
\langle P_{k,N},Q_{m,N}\rangle_{L^2(\mu_N)}=\delta_{km},\qquad 0\leq k,m\leq N-1,
\qe 
such that the joint probability distribution on $\R^N$ of  $x_1,\ldots,x_N$ reads
\eq
\label{DPPdistr2}
\frac{1}{N!}\det\Big[P_{k-1,N}(x_i)\Big]_{i,k=1}^N\det\Big[Q_{k-1,N}(x_i)\Big]_{i,k=1}^N\prod_{i=1}^N\mu_N(\di x_i).
\qe
If we introduce the (non-necessarily symmetric) kernel 
\eq
\label{kernelsplit}
K_N(x,y)=\sum_{k=0}^{N-1}P_{k,N}(x)Q_{k,N}(y),\qquad x,y\in\R,
\qe
 observe that the distribution \eqref{DPPdistr2} can be rewritten as
\eq
\label{DPPdistr}
\frac{1}{N!}\det\Big[K_N(x_i,x_j)\Big]_{i,j=1}^N\prod_{i=1}^N\mu_N(\di x_i)
\qe
and moreover  that the operator acting on $L^2(\mu_N)$ by
\eq
\pi_N:f(x)\mapsto\int K_N(x,y)f(y)\mu_N(\di y)
\qe
is a (non-necessarily orthogonal) bounded projection operator on an $N$-dimensional  subspace of $L^2(\mu_N)$. Conversely, any bounded projection operator with finite rank acting on some $L^2$ space induces a Biorthogonal Ensemble, as a consequence of the spectral theorem for compact operators. Thus, we understand from \eqref{DPPdistr2}--\eqref{DPPdistr} that Biorthogonal Ensembles matches with the class of determinantal point processes associated with  (non-trivial) bounded finite-rank projection operators; for further information on determinantal point processes, we refer to the references  \cite{HKPV,Joh,So}.

This type of asymmetric  distributions (in the sense that $K_N$ is not necessarily symmetric) has been firstly introduced by Borodin for the purpose of studying a one parameter deformation of classical Orthogonal Polynomial Ensembles \cite{Bo}. It moreover covers a large class of important determinantal processes, like Orthogonal Polynomial Ensembles or (mixed-type) Multiple Orthogonal Polynomial Ensembles; more information concerning these ensembles will be provided later. 

To the random variables $x_1,\ldots,x_N$, we associate their  average characteristic polynomial, 
\eq
\chi_N(z)=\mathbb E\left[\,\prod_{i=1}^N(z-x_i)\right],\qquad z\in\C,
\qe
where the expectation $\E$ refers to \eqref{DPPdistr2}, and we ask the following question: What is a sufficient condition so that the asymptotic distribution of the zeros of $\chi_N$ and the limiting distribution of the random variables $x_i$'s coincide as $N\rightarrow\infty$ ? More precisely, if one denotes by $z_1,\ldots,z_N$ the (non-necessarily real nor distinct) zeros of $\chi_N$ and introduces the  zero counting probability  measure 
\eq
\label{zerocountingmeasure}
\nu_N=\frac{1}{N}\sum_{i=1}^N\delta_{z_i},
\qe
the purpose of this work is to investigate the relation between the weak convergence of $\nu_N$ and the almost sure weak convergence of the empirical  measure of the determinantal point process, namely
\eq
\hat{\mu}^N=\frac{1}{N}\sum_{i=1}^N\delta_{x_i}.
\qe

It is for example known that the  eigenvalues of an $N\times N$ random matrix drawn from the GUE  form a Biorthogonal Ensemble, and that $\chi_N$ is the $N$-th monic (i.e. with leading coefficient one) Hermite polynomial. After an appropriate rescaling, the zero distribution $\nu_N$ converges weakly  towards the semi-circle distribution  as $N\rightarrow\infty$, and so is almost surely the spectral measure $\hat{\mu}^N$. There are several examples of  Biorthogonal Ensembles for which such a simultaneous convergence for $\hat\mu^N$ and $\nu_N$ is expected, but not proved yet.
 
The aim of this work is to provide a sufficient condition so that, as $N\rightarrow\infty$, the  convergence   of the moments of $\nu_N$ is equivalent to the almost sure convergence of the moments of $\hat{\mu}^N$  for a  large class of determinantal point processes, see Theorem \ref{mainth}.   We will actually  show that this condition implies the simultaneous moment convergence of $\nu_N$ and of the mean distribution $\mathbb E\big[\hat\mu^N\big]$, defined by $\mathbb E\big[\hat\mu^N\big](A)=\mathbb E\big[\hat\mu^N(A)\big]$ for any Borel set $A\subset\R$, and  moreover  forces the moments of $\hat\mu^N$ to concentrate around  their means at a  rate $N^{1+\epsilon}$, see Theorem \ref{concentration}. At this level of generality, the latter concentration result is new and may be of independent  interest. 

\subsection{Assumptions and statement of the results}
\label{DPP}

Given a sequence of Biorthogonal Ensembles  indexed by $N$ (the number of particles), which one can  parametrize by
 \eq
 \label{seqDPP}
\Bigg\{ \mu_N,\;\Big(P_{k,N}\Big)_{k=0}^{N-1},\; \Big(Q_{k,N}\Big)_{k=0}^{N-1}\Bigg\}_{N\geq 1},
\qe
we moreover assume  the following  structural assumption to hold  (throughout this paper we denote $\N=\{0,1,2,\ldots\}$).

\begin{assumption} 
\label{structassump}
\begin{itemize}\
\item[(a)]
For each $N$, the two families  $(P_{k,N})_{k=0}^{N-1}$ and  $(Q_{k,N})_{k=0}^{N-1}$ can be completed in two infinite biorthogonal families $(P_{k,N})_{k\in\N}$ and $(Q_{k,N})_{k\in\N}$ of $L^2(\mu_N)$, that is which satisfy 
\eq
\label{Biortho}
\langle P_{k,N},Q_{m,N}\rangle_{L^2(\mu_N)}=\delta_{km},\qquad  k,m \in\N.
\qe
\item[(b)]
There exists a sequence $(\frak q_N)_{N\geq 1}$ of integers having  sub-power growth, that is for every  $n\geq 1$,
\eq
\label{growthqN}
\frak q_N=o(N^{1/n})\qquad \mbox{as }N\rightarrow\infty,
\qe
 such that for all $k\in\N$,
\[
xP_{k,N}\in\Span\Big(P_{m,N}\Big)_{m=0}^{k+\frak q_N}.
\]
\end{itemize}
\end{assumption}

The next sections provide examples of Biorthogonal Ensembles which satisfy Assumption \ref{structassump}. Let $\mathbb P$ be the probability measure associated to the product probability space $\bigotimes_N(\R^N,\mathbb P_N)$, where $(\R^N,\mathbb P_N)$ is the probability space induced by \eqref{DPPdistr2}.  The central theorem of this work is the following.

\begin{theorem}
\label{mainth}
Assume there exists $\varepsilon>0$ such that   for every $n\geq 1$,
\eq
\label{sequence}
\max_{k,m\,\in\,\N \,:\;\left|\frac{k}{N}-1\right|\leq \,\varepsilon,\;\left|\frac{m}{N}-1\right|\leq \,\varepsilon}\left|\langle  xP_{k,N},Q_{m,N}\rangle_{L^2(\mu_N)}\right| = o(N^{1/n})
\qe
 as $N\rightarrow\infty$. Then, for all $\ell\in\N$,
\eq
\label{mainthcsq}
\lim_{N\rightarrow\infty}\left|\int x^\ell \hat{\mu}^N(\di x) -\int x^\ell \nu_N(\di x)\right|=0,\qquad \mathbb P\mbox{-almost surely}.
\qe
\end{theorem}

In practice, the sub-power growth condition \eqref{sequence} may be interpreted as the condition that a strong enough normalization for the $x_i$'s has been performed.

\begin{remark}
Assumption \ref{structassump} (a) and (b)  provide together  for each $N$ the (unique) decomposition
\eq
\label{decomposition}
xP_{k,N}=\sum_{m=0}^{k+\frak q_N}\langle xP_{k,N},Q_{m,N}\rangle_{L^2(\mu_N)}P_{m,N},\qquad k\in\N.
\qe 
Thus \eqref{sequence} is  a growth condition for the coefficients lying in a specific window of the infinite matrix (i.e.  operator on $\ell^2(\N)$) associated to the  operator $f(x)\mapsto xf(x)$ acting on  $\Span(P_{k,N})_{k\in\N}$. 
\end{remark}

Having in mind that probability measures on $\R$ with compact support are characterized by their moments, the following  consequence of Theorem \ref{mainth}  may be of use to obtain  almost sure convergence results. 
\begin{corollary}
\label{zero->part}
Under the assumption of Theorem \ref{mainth}, if there exists a probability measure $\mu^*$ on $\R$ characterized by its moments such that for all $\ell\in\N$,
\[
\lim_{N\rightarrow\infty}\int x^\ell \nu_N(\di x)=\int x^\ell \mu^*(\di x),
\]
then $\mathbb P$-almost surely $\hat\mu^N$ converges  weakly towards $\mu^*$ as $N\rightarrow\infty$.
\end{corollary}

Similarly, when one is interested in the limiting zero  distribution of $\chi_N$, the following corollary will be of help.

\begin{corollary}
\label{part->zero}
Under the assumption of Theorem \ref{mainth}, if
\begin{enumerate}
\item[{\rm (a)}]
 for all $N$ large enough $\chi_N$ has real zeros,
\item[{\rm (b)}]
   there exists a probability measure $\mu^*$ on $\R$ characterized by its moments such that for all $\ell\in\N$,
\eq
\label{momentconv}
\lim_{N\rightarrow\infty}\E\left[\,\int x^\ell \hat{\mu}^N(\di x)\right]=\int x^\ell \mu^*(\di x),
\qe
\end{enumerate}
then $\nu_N$ converges weakly towards $\mu^*$ as $N\rightarrow\infty$.
\end{corollary}

As an example of application, we will obtain in Section \ref{MOP} a description for the limiting zero distribution of multiple Hermite and multiple Laguerre polynomials, see Theorems \ref{ThHermite} and \ref{ThLaguerre}. At the best knowledge of the author, this is the first time that a description of these zero limiting  distributions is provided in such a level of generality.

\begin{remark}  Although it is not hard to see from our proofs that Theorem \ref{mainth} continues to hold for determinantal point processes on $\C$ (with the introduction of complex conjugations where needed), Corollaries \ref{zero->part} and \ref{part->zero} are not true in the complex setting. Indeed, consider the eigenvalues of an $N\times N$ unitary matrix distributed according to the Haar measure, which are known to form an OP Ensemble  on the unit circle with respect to its uniform measure. We have $\chi_N(z)=z^N$, and thus  $\nu_N=\delta_0$ for all $N$, but the spectral measure $\hat\mu^N$ is known to converge towards the uniform distribution on the unit circle as $N\rightarrow\infty$.
\end{remark}

On the road to establish Theorem \ref{mainth}, we prove the following  variance decay  which  basically allows to extend the mean convergence of the moments of $\hat\mu^N$ to the almost sure one, by combining the Chebyshev inequality and  the Borel-Cantelli lemma. 
 
\begin{theorem}
\label{concentration}
Under the assumptions of Theorem \ref{mainth}, for every $0<\alpha<1$ and  any $\ell\in\N$, there exists $C_{\alpha,\ell}$ independent of $N$ such that
\eq
\label{concentrationcsq}
\mathbb V{\rm ar}\left[\,\int x^\ell \hat\mu^N(\di x)\right]\leq \frac{C_{\alpha,\ell}}{N^{1+\alpha} }.
\qe
If moreover $\frak q_N$ and the left-hand side of \eqref{sequence} are bounded (seen as sequences of the parameter $N$), then \eqref{concentrationcsq} also holds for $\alpha=1$.
\end{theorem}

Before to provide  proofs for Theorem \ref{mainth} and \ref{concentration}, we now describe    a few Biorthogonal Ensembles which are concerned by our results.

\subsection{Orthogonal Polynomial Ensembles} 
\label{OP1}
Examples of  Orthogonal Polynomial (OP) Ensembles are provided by  eigenvalue distributions of  unitary invariant Hermitian random matrices, including the GUE,  Wishart and Jacobi matrix models; they also arise from non-intersecting diffusion processes starting and ending at the origin. In the latter examples, $\mu_N$ has a density with respect to the Lebesgue measure. They moreover play a key role in the resolution of several problems from asymptotic combinatorics, such as the problem of the longest increasing subsequence of a random  permutation, the shape distribution of large Young diagrams, the  random tilings of an Aztec diamond (resp. hexagone) with dominos (resp. rhombuses). This time $\mu_N$ is a discrete measure. For further information, see \cite{Ko,Joh1,Joh2} and references therein.  

The joint probability distribution of real random variables $x_1,\ldots,x_N$ drawn from an OP Ensemble reads
\[
\frac{1}{Z_N}\prod_{1\leq i < j \leq N}\big|x_i-x_j\big|^2\prod_{i=1}^N\mu_N(\di x_i),
\]
where $Z_N$ is a  normalization constant and $\mu_N$ is a measure on $\R$ having all its moments. One can rewrite that distribution in the form \eqref{DPPdistr2} by taking for $P_{k,N}=Q_{k,N}$ the $k$-th orthonormal polynomial for $\mu_N$. The associated operator is then the orthogonal projection onto the subspace of $L^2(\mu_N)$ of polynomials having degree at most $N-1$. Thus, OP Ensembles satisfy Assumption \ref{structassump} with $\frak q_N=1$. 

An important observation, provided by a classical integral representation  for OPs attributed to Heine, see e.g. \cite[Proposition 3.8]{Dei}, is that the average characteristic polynomial $\chi_N$ associated to an OP Ensemble equals the $N$-th monic OP with respect to $\mu_N$. Since OPs are known to have real zeros, $\nu_N$ is thus supported on  $\R$.

As we shall recall in Section \ref{proof}, the mean distribution $\mathbb E\big[\hat\mu^N\big]$ of  a determinantal point process  reads $\frac{1}{N}K_N(x,x)\mu_N(\di x)$. Quite remarkably, it turns out that in the case of OP Ensembles, the  convergence  of the mean distribution  has been investigated in the approximation theory literature, where it is referred as the weak convergence of the Christoffel-Darboux kernel. Indeed, recall that OPs satisfy the three-term recurrence relation
\begin{align*}
xP_{k,N} \;& =\;a_{k+1,N}P_{k+1,N}+b_{k,N}P_{k,N}+a_{k,N}P_{k-1,N},\qquad k\geq 1,\nonumber\\
xP_{0,N} \;& =\;a_{1,N}P_{1,N}+b_{0,N}P_{0,N}.
\end{align*}
Using the determinantal point processes terminology, Nevai \cite{Nev} and Van Assche \cite{VA} actually proved   that if $a_{N,N} =o(N^{1/2})$, then for any continuous and bounded function $f$ on $\R$,
\eq
\label{weakconvOP}
\lim_{N\rightarrow\infty}\left|\E\left[\,\int f(x)\hat\mu^N(\di x)\right]-\int f(x)\nu_N(\di x)\right| =0.
\qe
Nevertheless, their proofs  involve the Gaussian quadrature associated to OPs, an argument which does not seem to be generalizable to more general Biorthogonal Ensembles.  More recently, and in the case where  the supports of the  measures $\mu_N$ are uniformly bounded, Simon \cite{S1} proved the simultaneous moment convergence of $\mathbb E\big[\hat\mu^N\big]$  and  $\nu_N$ by means of  elegant operator-theoretic arguments, which have been of inspiration for this work. The (little) novelty of Theorem \ref{mainth} for OP Ensembles is  to show that \eqref{weakconvOP} also holds when $f$ is polynomial, provided there exists $\varepsilon>0$ such that both
\[
\max_{k\in\,\N :\;\left|\frac{k}{N}-1\right|\leq\, \varepsilon}\big|a_{k,N}|\qquad \mbox{ and } \quad \max_{k\in\,\N :\;\left|\frac{k}{N}-1\right|\leq\, \varepsilon}\big| b_{k,N}\big|
\]
 have a sub-power growth as $N\rightarrow\infty$.

Our result also strengthens the mean convergence to the almost sure one, but let us mention that by using the Christoffel-Darboux formula for the kernel of OP Ensembles, the variance decay \eqref{concentrationcsq} can be alternatively obtained with our growth assumption from an easy adaptation of the proof of \cite[Theorem 4.3.1.(ii)]{PS}. The advantage of our approach here is   we do not use the Christoffel-Darboux formula, so that it applies to more general Biorthogonal Ensembles where such a formula is not available. 

Let us also mention  that Breuer and Duits  recently established in \cite{BD}  that if $a_{N,N}=o(N^{1/2})$, for any continuous and bounded function $f$ the concentration of  $\int f(x)\hat\mu^N(\di x)$ around its mean actually happens  at an exponential rate.  Their proof is based on the Laplace transform approach for concentration inequalities which provides a much more accurate upper bound than the  Chebyshev inequality we  use in our proof.

\subsection{One-parameter deformation of OP Ensembles}
Borodin introduced the concept of Biorthogonal Ensembles in \cite{Bo} to study the following one parameter deformation of an OP Ensemble
\eq
\label{biodistr}
\frac{1}{Z_N}\prod_{1\leq i < j \leq N}(x_i-x_j)(x_i^\theta-x_j^\theta)\prod_{i=1}^N\mu_N(\di x_i),
\qe
where $\theta$ is a fixed positive real number and, when $\theta$ is non-integer, we assume for the sake of simplicity that $\mu_N$ is a measure supported on $\R_+$.  The motivation to study such ensembles arises from Muttalib's work on  modeling disordered conductors in the metallic regime \cite{Mut}. 

In a similar fashion than for OP Ensembles, for any $N$ one can rewrite \eqref{biodistr} in the form \eqref{DPPdistr2} with $P_{k,N}$ a polynomial of degree $k$, and $Q_{k,N}(x)=\widetilde{Q}_{k,N}(x^\theta)$ where $\widetilde{Q}_{k,N}$ is a polynomial of degree $k$, such that the $P_{k,N}$'s and the $Q_{k,N}$'s are biorthogonal; they are called biorthogonal polynomials in the literature, see the numerous references in \cite{Bo}. These ensembles  satisfy Assumption Ê\ref{structassump} with $\frak q_N=1$.

Equivalently, one can express the biorthogonal relations between the $P_{k,N}$'s and $Q_{k,N}$'s  by the relations 
\[
\langle P_{k,N}, x^{\theta j}\rangle_{L^2(\mu_N)}=0,\quad  \langle Q_{k,N}, x^{ j}\rangle_{L^2(\mu_N)}=0, \qquad 0\leq j \leq k-1,\quad k\geq 1.
\]
It is  then easy to show by using similar arguments than in the proof of  \cite[Proposition 3.8]{Dei} that the average characteristic polynomial $\chi_N$ satisfies $\langle \chi_N, x^{\theta j}\rangle_{L^2(\mu_N)}=0$ for all $0\leq j \leq N-1$ and thus equals $P_{N,N}$ up to a multiplicative constant. 

Our results seem to be completely new for such ensembles; no Christoffel-Darboux type formula is available for the kernel $K_N$ in the general $\theta>0$ case.
 
\subsection{Multiple Orthogonal Polynomial Ensembles}
\label{MOP1}

Firstly introduced by Bleher and Kuijlaars \cite{BK} to describe the eigenvalue distribution of an additive perturbation of the GUE, breaking the unitary invariance, Multiple Orthogonal Polynomial (MOP) Ensembles show up in several perturbed matrix models \cite{BDK,BK2,DF}, in multi-matrix models \cite{Ku3, DK, DKM} as well, and in non-intersecting diffusion processes with arbitrary prescribed starting points and ending at the origin \cite{KMFW}.  For general   presentations, see \cite{Ku1,Ku2} and the references therein.

The joint distribution of real random variables $x_1,\ldots,x_N$ distributed according to a MOP Ensemble has the following form
\eq
\label{MOPdensity}
\frac{1}{Z_{\bv n , N}}\prod_{1\leq i<j\leq N}\big(x_j-x_i\big)\det\left[\begin{matrix} \Big\{x_j^{i-1}w_{1,N}(x_j)\Big\}_{i,j=1}^{n_1,\,N} \\ \vdots \\ \Big\{x_j^{i-1}w_{r,N}(x_j)\Big\}_{i,j=1}^{n_r,\,N}
\end{matrix}\right]\prod_{i=1}^N\mu_N(\di x_i),
\qe
where $\mu_N$ is a measure on $\R$ having all its moments,  $Z_{\bv n,N}$ is a  normalization constant and the weights $w_{1,N},\ldots,w_{r,N}\in L^2(\mu_N)$ are such that \eqref{MOPdensity} is indeed a probability distribution. The multi-index $\bv n =(n_1,\ldots,n_r)\in\N^r$ depends on $N$ and satisfies $\sum_{i=1}^rn_i=N$. Note that we recover OP Ensembles by taking $r=1$.

It turns out one can rewrite \eqref{MOPdensity} in the form \eqref{DPPdistr2} where the $P_{k,N}$'s are monic polynomials with $\deg P_{k,N}=k$,   the $Q_{k,N}$'s are (non-necessarily polynomial) $L^2(\mu_N)$-functions biorthogonal to the $P_{k,N}$'s, and MOP Ensembles satisfies Assumption \ref{structassump} with $\frak q_N=1$, see Section \ref{MOP}.

Kuijlaars  \cite[Proposition 2.2]{Ku1} established that the average characteristic polynomial $\chi_N$ associated to \eqref{MOPdensity} is the $\bv n$-th  (type II) MOP associated with the weights $w_{i,N}$,  $1\leq i \leq r$, and the measure $\mu_N$,  see Definition \ref{defMOPII}.

The simultaneous convergence of the empirical measure $\hat\mu^N$  and the zero distribution $\nu_N$ of the associated MOPs is expected for several MOP Ensembles. It is for example the case for non-intersecting squared Bessel paths with positive starting point and ending at the origin. Indeed, for this MOP Ensemble $\mathbb E\big[\hat\mu^N\big]$ converges towards a limiting measure described in terms of the solution of a vector equilibrium problem, see \cite[Theorem 2.4 and Appendix]{KMFW}, and the limit of $\nu_N$ benefits from the same description \cite{KR}.  The same situation holds in the two-matrix model with quartic/quadratic potentials, by combining the works \cite{DK} and \cite{DGK}. For the non-intersecting squared Bessel paths model, which is equivalent to a non-centered complex Wishart matrix model, the almost sure convergence of $\hat\mu^N$ towards the solution of the vector equilibrium problem  has recently been obtained as a consequence of a stronger large deviation principle \cite{HK}. For the two matrix model,  to prove a large deviation upper bound involving a rate function associated to a vector equilibrium problem is still an open problem, see \cite{DKM2} for further discussion.  For these two MOP Ensembles, the  almost sure simultaneous convergence of $\hat\mu^N$ and $\nu_N$ follows from Theorem \ref{mainth}, since the asymptotics of the recurrence coefficients $\langle xP_{k,N},Q_{m,N}\rangle_{L^2(\mu_N)}$'s are actually explicitly described  in \cite[(1.11)]{KR} and \cite[Theorem 5.2]{DGK} respectively. An other example of MOP Ensemble where the same conclusion holds is provided by \cite{BDK2}, in relation with the six-vertex model. 

\begin{remark} 
\label{RH}
Let us stress that our results for MOP Ensembles combine nicely with a Deift-Zhou steepest descent analysis. Indeed, it is known for such ensembles that one can represent $K_N$ in terms of the solution of a Riemann-Hilbert problem, see \cite{Ku1}. This, in principle, allows to use the Deift-Zhou steepest descent method, which yields a precise asymptotic description of $K_N$, and related quantities. In particular, the $\langle xP_{k,N},Q_{m,N}\rangle_{L^2(\mu_N)}$'s can be expressed in terms of the solution of the Riemann-Hilbert problem (see \cite[Section 5]{GKVA}) and a control of their growth would follow from that steepest descent analysis (alternatively,  a control of the growth of the nearest neighbor recurrence coefficients is also sufficient, as explained in Section \ref{MOP}). Such an asymptotic analysis also typically provides the locally uniform convergence and tail estimates for $K_N$ as $N\rightarrow\infty$,  from which would follow \eqref{momentconv}, and where the limiting measure $\mu^*$ has in general compact support. In most cases, the zeros of $\chi_N$ are real; this is always true for  important subclasses of MOPs like Angelesco or AT systems \cite{Is}. Thus, if one assumes the latter to be true, the combination of a successful Deift-Zhou steepest descent analysis together with  Corollary \ref{part->zero} and Theorem \ref{concentration}  would  provide the almost sure weak convergence of the empirical measure $\hat\mu^N$, and moreover the weak convergence  of the zero distribution $\nu_N$ of the MOPs towards $\mu^*$, without extra effort. \end{remark}

\begin{remark}  As we have seen, OP Ensembles, their $\theta$-deformation, and MOP Ensembles all satisfy Assumption \ref{structassump} with $\frak q_N=1$. A class of determinantal point processes which satisfy this assumption but for which $\frak q_N$ may grow is provided by mixed-type MOP Ensembles (where $P_{k,N}$'s are no longer polynomials), originally introduced by Daems and Kuijlaars to describe non-intersecting Brownian bridges with arbitrary starting and ending points \cite{DeK}. Delvaux showed that the average characteristic polynomial $\chi_N$ is  in this case a mixture of MOPs \cite{Del}.
\end{remark}

The rest of this work is structured as follows.  In Section \ref{proof}, we establish Theorems \ref{mainth} and \ref{concentration}. In Section \ref{MOP}, after a quick introduction to MOPs, we use  Corollary \ref{part->zero} and Voiculescu's theorems in order to identify the limiting zero distribution of the multiple Hermite and multiple Laguerre polynomials in terms of free convolutions, and moreover derive algebraic equations for their Cauchy-Stieltjes transform.  
\section{Proof of the main theorems}
\label{proof}
 
In a first step to establish Theorems \ref{mainth} and \ref{concentration}, we express all the quantities of interest in terms of traces of appropriate operators; this is a usual move in the theory of determinantal point processes.

\subsection{Step 1 : Tracial representations}
\label{trace}

Consider a   determinantal point process associated to the rank $N$ bounded projector $\pi_N$ acting on $L^2(\mu_N)$ with  kernel $K_N$ given by  \eqref{kernelsplit}, so that 
\[
{\rm Im}(\pi_N)=\Span\Big(P_{k,N}\Big)_{k=0}^{N-1}, \qquad {\rm Ker}(\pi_N)^{\perp}=\Span\Big(Q_{k,N}\Big)_{k=0}^{N-1}.
\] 
The usual definition of a determinantal point process, see e.g. \cite{Joh}, provides for any $n\geq 1$ and any Borel function $f:\R^n\rightarrow\R$ the identity
\begin{multline}
\label{defDPP}
\E\Bigg[\sum_{i_1\,\neq\,\cdots\, \neq\,i_n} f(x_{i_1},\ldots,x_{i_n})\Bigg]\\
=\int f(x_1,\ldots,x_n)\det\Big[K_N(x_i,x_j)\Big]_{i,j=1}^n\prod_{i=1}^n\mu_N(\di x_i),
\end{multline}
where the summation concerns all pairwise distinct indices taken from $\{1,\ldots,N\}$. 

Let $M$ be the  operator acting on $L^2(\mu_N)$ by  
\eq
Mf(x)=xf(x).
\qe
Then, it is  standard to show that the following identity holds.
 
\begin{lemma} 
\label{Etrace}
For any $\ell\in\N$,
\[
\mathbb E\left[\,\int x^\ell \hat\mu^N(\di x)\right]  = \frac{1}{N}{\rm Tr}\big(\pi_N M^\ell\pi_N\big).
\]
\end{lemma}

\begin{proof}
By using \eqref{defDPP} with $n=1$, \eqref{kernelsplit} and the biorthogonality relations \eqref{biortho}, we obtain
\begin{align*}
\E\left[\sum_{i=1}^Nx_i^\ell\right]  & = \sum_{k=0}^{N-1}\int x^\ell P_{k,N}(x)Q_{k,N}(x)\mu_N(\di x)\\
 & =  \sum_{k=0}^{N-1}\langle (\pi_N M^\ell \pi_N) P_{k,N} ,  Q_{k,N}\rangle_{L^2(\mu_N)} \\
 & =  \Tr\big(\pi_N M^\ell \pi_N\big).
\end{align*} 
\end{proof}

We also represent the variance of the moments in a similar fashion. 

\begin{lemma} 
\label{Vartrace}
For any $\ell\in\N$,
\[
\mathbb V{\rm ar}\left[\,\int x^\ell \hat\mu^N(\di x)\right] = \frac{1}{N^2}\Big( {\rm Tr}\big(\pi_N M^{2\ell}\pi_N\big)-{\rm Tr}\big(\pi_N M^\ell\pi_N M^\ell\pi_N\big)\Big).
\]
\end{lemma}

\begin{proof}
We write
\[
\label{defvar}
\mathbb V{\rm ar}\left[\sum_{i=1}^Nx_i^\ell\right]  = \mathbb E\left[\sum_{i=1}^Nx_i^{2\ell}\right]+\mathbb E\Bigg[\sum_{i\,\neq \,j}x_i^\ell x_j^\ell\Bigg]-\left(\mathbb E\left[\sum_{i=1}^Nx_i^\ell\right] \right)^2
\]
in order to  obtain, thanks to  \eqref{defDPP} with $n=2$ and Lemma \ref{Etrace},
\[
\mathbb V{\rm ar}\left[\sum_{i=1}^Nx_i^\ell\right] = {\rm Tr}\big(\pi_N M^{2\ell}\pi_N\big) - \iint x^\ell y^\ell K_N(x,y)K_N(y,x)\mu_N(\di x)\mu_N(\di y).
\]
Finally, observe that
\begin{align*}
 & \iint x^\ell y^\ell K_N(x,y)K_N(y,x)\mu_N(\di x)\mu_N(\di y) \\
& = \quad \sum_{k=0}^{N-1}\int x^\ell\left(\,\int K_N(x,y)y^\ell P_{k,N}(y)\mu_N(\di y)\right)Q_{k,N}(x)\mu_N(\di x)\\
& = \quad \sum_{k=0}^{N-1}\langle \pi_N M^\ell \pi_N M^\ell \pi_N P_{k,N} , Q_{k,N} \rangle_{L^2(\mu_N)}\\
& = \quad {\rm Tr}\big(\pi_N M^\ell\pi_N M^\ell\pi_N\big)
\end{align*}
to complete the proof.
\end{proof}

We now check that the average characteristic polynomial $\chi_N$ equals the characteristic polynomial of the operator $\pi_N M\pi_N$ acting on $\im(\pi_N)$.

\begin{proposition} 
\label{chiNrepresentation}
If $\det$ stands for the determinant of endomorphisms of $\im(\pi_N)$, then
\[
\chi_N(z)=\det\big(z-\pi_N M\pi_N\big), \qquad z\in\C.
\]
\end{proposition}

\begin{proof}
On the one hand,  Vieta's formulas provide 
\[
\mathbb E\left[\,\prod_{i=1}^N(z-x_i)\right] =z^N+\sum_{n=1}^N\frac{1}{n!}(-1)^nz^{N-n}\;\E\left[\sum_{i_1\,\neq\,\cdots\,\neq \,i_n }x_{i_1}\cdots \,x_{i_n}\right]
\]
and \eqref{defDPP}  yields for any $1\leq n \leq N$
\[
\E\left[\sum_{i_1\,\neq\,\cdots\,\neq \,i_n }x_{i_1}\cdots \,x_{i_n}\right]
=\int\det\Big[x_jK(x_i,x_j)\Big]_{i,j=1}^n\prod_{i=1}^n\mu_N(\di x_i).
\]
On the other hand, since $\pi_N M\pi_N$ is an integral operator acting on $\im(\pi_N)$ with kernel $(x,y)\mapsto yK_N(x,y)$, the  Fredholm's expansion, see e.g. \cite{GGK}, reads 
\[
\det\big(z-\pi_N M\pi_N\big)  
= z^N+\sum_{n=1}^N\frac{1}{n!}(-1)^nz^{N-n}\int\det \Big[x_jK_N(x_i,x_j)\Big]_{i,j=1}^n\prod_{i=1}^n\mu_N(\di x_i),
\]
from which Proposition \ref{chiNrepresentation} follows.
\end{proof}

The next immediate corollary  will be of important use in what follows. 

\begin{corollary}
\label{tracezero}
For any $\ell\in\N$, 
\[
\int x^\ell \nu_N(\di x)=\frac{1}{N}\tr\big((\pi_N M\pi_N)^\ell\big).
\]
\end{corollary}

The second step  is to rewrite the traces in terms of weighted lattice paths.

\subsection{Step 2 : Lattice paths representations}
\label{latpath}
We introduce for each $N$ the oriented graph $\mathcal G_N= (\mathcal V_N , \mathcal E_N)$ having  $\mathcal V_N=\N^2$ for vertices and for  edges
\[
\mathcal E_N =\Big\{(n,k)\rightarrow (n+1,m), \qquad  n,k\in\N,\quad 0\leq m \leq k+\frak q_N \Big\}.
\] 
To each edge  is associated a weight 
\[
w_N\Big((n,k)\rightarrow (n+1,m)\Big)= \langle xP_{k,N} , Q_{m,N} \rangle_{L^2(\mu_N)},
\]
and  the weight of a finite length oriented path $\gamma$  on $\mathcal G_N$ is defined as the product of the weights of the edges contained in $\gamma$, namely
\eq
\label{defweight}
w_N(\gamma)=\prod_{e\in\mathcal E_N : \;e\subset\gamma}w_N(e).
\qe
Then the following holds.

\begin{lemma} For any $\ell\in\N$,
\label{latpath1}
\eq
\mathbb E\left[\,\int x^\ell \hat\mu^N(\di x)\right] =\frac{1}{N}\sum_{k=0}^{N-1}\sum_{\gamma : (0,k)\rightarrow (\ell,k)}w_N(\gamma),
\qe
where the rightmost summation concerns all the oriented paths on $\mathcal G_N$ starting from $(0,k)$ and ending at $(\ell,k)$.
\end{lemma}

\begin{proof} It  follows inductively on $\ell$  from \eqref{decomposition} and the definition \eqref{defweight} that

\eq
\label{AA}
(\pi_NM^\ell\pi_N)P_{k,N}=\sum_{m=0}^{N-1}\left(\sum_{\gamma : (0,k)\rightarrow (\ell,m)}w_N(\gamma)\right)P_{m,N},\qquad \ell,k\in\N.
\qe
Thus, we obtain from the biorthogonality relations \eqref{biortho}
\begin{align}
\label{AAA}
{\rm Tr}\big(\pi_N M^\ell\pi_N\big) = &\; \sum_{k=0}^{N-1}\langle (\pi_NM^\ell\pi_N)P_{k,N} , Q_{k,N}\rangle_{L^2(\mu_N)}\nonumber\\
= &\; \sum_{k=0}^{N-1}\sum_{\gamma : (0,k)\rightarrow (\ell,k)}w_N(\gamma),
\end{align}
and Lemma \ref{latpath1}  follows from Lemma \ref{Etrace}.

\end{proof}

Next, we introduce 
\eq
\label{DN}
D_N=\Big\{(n,m)\in\N^2 : \; m \geq N\Big\}
\qe
and obtain a similar  representation for the moments of $\nu_N$.

\begin{lemma}  
\label{latpath2}
For any $\ell\in\N$,
\eq
\int x^\ell \nu_N(\di x)=\frac{1}{N}\sum_{k=0}^{N-1}\sum_{\gamma : (0,k)\rightarrow (\ell,k),\;\gamma\cap D_N=\emptyset}w_N(\gamma).
\qe
\end{lemma}

\begin{proof} Similarly than for \eqref{AA}, we have
\eq
(\,\underbrace{\pi_NM \cdots \pi_NM}_{\ell}\pi_N )P_{k,N}=\sum_{m=0}^{N-1}\left(\sum_{\gamma : (0,k)\rightarrow (\ell,m),\;\gamma\cap D_N=\emptyset}w_N(\gamma)\right)P_{m,N},\qquad \ell,k\in\N.
\qe
Since $\pi_N^2=\pi_N$, this yields
\begin{align}
{\rm Tr}\big((\pi_N M\pi_N)^\ell\big) = &\; \sum_{k=0}^{N-1}\langle (\,\underbrace{\pi_NM \cdots \pi_NM}_{\ell}\pi_N )P_{k,N} , Q_{k,N}\rangle_{L^2(\mu_N)}\nonumber\\
= &\; \sum_{k=0}^{N-1}\sum_{\gamma : (0,k)\rightarrow (\ell,k), \;\gamma\cap D_N=\emptyset}w_N(\gamma)
\end{align}
and thus Lemma \ref{latpath2}, because of Corollary \ref{tracezero}.
\end{proof}

If we denote by $\gamma(m)$ the ordinate of a path $\gamma$ at abscissa $m$, then we can  represent the variance of the moments of $\hat\mu^N$ in a similar fashion. 

\begin{lemma} 
\label{latpath3}
For any $\ell\in\N$,
\eq
\label{variancepaths}
\mathbb V{\rm ar}\left[\,\int x^\ell \hat\mu^N(\di x)\,\right]=\frac{1}{N^2}\sum_{k=0}^{N-1}\sum_{\gamma : (0,k)\rightarrow (2\ell,k),\; \gamma(\ell)\geq N}w_N(\gamma).
\qe
\end{lemma}

\begin{proof} 
We have already shown in \eqref{AAA} that
\eq
\label{BB}
{\rm Tr}\big(\pi_N M^{2\ell}\pi_N\big)=\sum_{k=0}^{N-1}\sum_{\gamma : (0,k)\rightarrow (2\ell,k)}w_N(\gamma).
\qe
Since 
\[
(\pi_NM^\ell\pi_NM^\ell\pi_N )P_{k,N}=\sum_{m=0}^{N-1}\left(\sum_{\gamma : (0,k)\rightarrow (\ell,m),\;\gamma(\ell)<N}w_N(\gamma)\right)P_{m,N},\qquad \ell,k\in\N,
\]
we moreover obtain  
\eq
\label{BBB}
{\rm Tr}\big(\pi_N M^\ell\pi_N M^\ell\pi_N\big)=\sum_{k=0}^{N-1}\sum_{\gamma : (0,k)\rightarrow (2\ell,k),\; \gamma(\ell)<N}w_N(\gamma).
\qe
Lemma \ref{latpath3} is then a consequence of Lemma \ref{Vartrace} and \eqref{BB}--\eqref{BBB}.
\end{proof}
We are now in position to complete the proofs of Theorems \ref{mainth} and \ref{concentration}.
\subsection{Step 3 : Majorations and conclusions}
\label{prooftheorems}

Let us first provide a proof for Theorem \ref{mainth} assuming that Theorem \ref{concentration} holds.
\begin{proof}[Proof of Theorem \ref{mainth}]
It is enough to prove that  for any given $\ell\in\N$  
\eq
\label{toprove}
\lim_{N\rightarrow\infty}\left|\mathbb E\left[\,\int x^\ell \hat\mu^N(\di x)\right] -\int x^\ell \nu_N(\di x)\right|=0,
\qe
since \eqref{mainthcsq} would then follow from Theorem \ref{concentration}, together with  the Chebyshev inequality and  the Borel-Cantelli lemma. As a consequence of  Lemmas \ref{latpath1} and \ref{latpath2}, we obtain
\eq
\label{mainth1}
\mathbb E\left[\,\int x^\ell \hat\mu^N(\di x)\right] -\int x^\ell \nu_N(\di x)= \frac{1}{N}\sum_{k=0}^{N-1}\sum_{\gamma : (0,k)\rightarrow (\ell,k),\; \gamma\cap D_N\neq\emptyset}w_N(\gamma).
\qe
Since by following an edge  of $\mathcal G_N$ one increases the ordinate by  at most $\frak q_N$, the  rightmost sum of \eqref{mainth1} will bring  null contribution if $k$ is strictly less that $N-\ell\frak  q_N $. Observe moreover that the vertices explored by any path $\gamma$ going from $(0,k)$ to $(\ell,k)$ for some $N-\ell\frak q_N\leq k\leq N-1$ such that $\gamma \cap D_N\neq \emptyset$ form a subset of 
\[
\Big\{ (n,m)\in\N^2 : \quad 0 \leq n\leq \ell,\quad N-\ell\frak q_N \leq m < N+\ell\frak q_N \Big\}.
\]
As a consequence, if one roughly bounds from above  the number of such paths by $(2\ell\frak q_N)^\ell$, one obtains from \eqref{mainth1} that 
\begin{multline}
\label{upperb}
\left|\,\mathbb E\left[\,\int x^\ell \hat\mu^N(\di x)\right] -\int x^\ell \nu_N(\di x)\,\right|\\
\leq \frac{ \left(2\ell\frak q_N\right)^\ell}{N}\max_{k,m\in\N : \;  \left|\frac{k}{N}-1\right| \leq \frac{\ell\frak q_N }{N},\; \left|\frac{m}{N}-1\right| \leq \frac{\ell\frak q_N}{N}}\left| \langle xP_{k,N} , Q_{m,N}\rangle_{L^2(\mu_N)}\right|^\ell.
\end{multline}
It then follows from \eqref{upperb} together with the growth assumptions \eqref{growthqN} and \eqref{sequence} that \eqref{toprove} holds, and the proof of Theorem \ref{mainth} is therefore complete up to the proof of Theorem \ref{concentration}.
\end{proof}

We now prove Theorem \ref{concentration} by using similar arguments than in the proof of Theorem \ref{mainth}.

\begin{proof}[Proof of Theorem \ref{concentration}]

Again, because following an edge of $\mathcal G_N$  increases the ordinate of at most $\frak q_N$, the rightmost sum of \eqref{variancepaths} brings zero contribution except when  $k\geq N-\ell\frak q_N $. Observe also that the vertices explored by any path $\gamma$ going from $(0,k)$ to $(2\ell,k)$ for some $N-\ell\frak q_N\leq k\leq N-1$ and satisfying $\gamma(\ell)\geq N$  form a subset of 
\[
\Big\{ (n,m)\in\N^2 : \quad 0 \leq n\leq 2\ell,\quad N-2\ell\frak q_N \leq m < N+2\ell\frak q_N \Big\}.
\]
As a consequence, we obtain from Lemma \ref{latpath3} the (rough) upper-bound
\begin{multline}
\label{CCC}
\mathbb V{\rm ar}\left[\,\int x^\ell \hat\mu^N(\di x)\,\right]\\
\leq \frac{(4\ell\frak q_N )^{2\ell}}{N^2}\max_{k,m\in\N : \;  \left|\frac{k}{N}-1\right| \leq \frac{2\ell\frak q_N }{N},\; \left|\frac{m}{N}-1\right| \leq \frac{2\ell\frak q_N}{N}}\left| \langle xP_{k,N} , Q_{m,N}\rangle_{L^2(\mu_N)}\right|^{2\ell}.
\end{multline}
Using the sub-power growth/boundedness assumptions on $\frak q_N$ and on the left-hand side of \eqref{sequence},  Theorem \ref{concentration}  follows.

\end{proof}

\section{Application to multiple orthogonal polynomials}
\label{MOP}

MOPs have been introduced in the context of the Hermite-Pad\'e approximation of Stieltjes functions, which was itself first motivated by  number theory after  Hermite's  proof of the transcendence of $e$, or Ap\'ery's proof of the irrationality of $\zeta(2)$ and $\zeta(3)$, see  \cite{VA99} for a survey. For our purpose here, we will focus on the so-called type II MOPs, for which the zeros are of important interest since they are  the poles of the rational approximants provided by the Hermite-Pad\'e theory. These polynomials generalize orthogonal polynomials in the sense that we consider more than one measure of orthogonalization, and a class of classical MOPs such as  multiple versions of the Hermite, Laguerre,  Jacobi, Charlier, Meixner, etc,  polynomials  emerged \cite{CVA0, ABVA, ACVA}. They are already the subject of many works where they are studied as special functions; we refer to the monograph \cite{Is} for further information. 

It turns out that even for the multiple Hermite or multiple  Laguerre polynomials,  no general description of the limiting zero distribution seems yet available in the literature. To motivate further the importance of an asymptotic description for the zeros, let us mention that they play an important role in the strong asymptotic  for MOPs. For example, the work \cite{LW} of Lysov and Wielonsky deals with the strong asymptotics of multiple Laguerre polynomials in the case where $r=2$. A key ingredient in their analysis is the a priori knowledge of an algebraic equation for the Cauchy-Stieltjes transform of the limiting zero distribution (that they denote $\psi(z)$, up to a trivial rescaling), see equation \cite[(1.4)]{LW}. Our purpose in this section is to show that our results allow to transport the powerful technology developed in free probability  to the description of such zero distributions. In particular, we obtain algebraic equations for the Cauchy-Stieltjes transform of the multiple Hermite and Laguerre polynomials in the general case where $r\geq 2$.  

Let us first introduce MOPs.

\subsection{Multiple orthogonal polynomials}
\label{MOPdef}
Let $\mu$   be a Borel measure on $\R$  with infinite support and having all its moments. Consider  $r\geq 1$ pairwise distinct functions $w_{1},\ldots,w_{r}$ in $L^2(\mu)$.  

\begin{definition}
\label{defMOPII} Given a  multi-index $\bv n=(n_1,\ldots,n_r)\in\N^r$, the $\bv n$-th (type  \2) MOP  associated to the weights $w_1,\ldots,w_r$ and the measure $\mu$ is the unique monic polynomial $P_{\bv n}$ of degree $n_1+\cdots +n_r$ which satisfies the  orthogonality relations
\begin{align}
\label{MOPIIortho}
\int x^kP_{\bv n}(x) w_1(x)\mu(\di x)=0, &\qquad 0\leq k \leq n_1-1,\nonumber\\
\vdots\qquad\qquad  &\qquad \qquad\vdots  \\
\int x^kP_{\bv n}(x) w_r(x)\mu(\di x)=0, &\qquad 0\leq k \leq n_r-1 \nonumber.
\end{align}

\end{definition}

Note that the existence/uniqueness of the $\bv n$-th MOP is not automatic, and depends on whether the  system of linear equations  \eqref{MOPIIortho}  admits a unique solution. We say that a multi-index $\bv n$ is normal if it is indeed  the case. Since by taking $r=1$ we clearly recover OPs,  we shall assume $r\geq 2$ in what follows.

Let   $(\bv n^{(N)})_{N\in\N}=(n_1^{(N)},\ldots,n_r^{(N)})_{N\in\N}$ be a sequence of normal multi-indices
which satisfies the following path-like structure.
 \begin{enumerate}
 \item[(a)]
For every $N\in\N$, 
\[
\sum_{i=1}^Nn_i^{(N)}=N.
\]   
\item[(b)]
For every $N\in\N$ and $1\leq i\leq r$, 
\[ 
n ^{(N+1)}_i\geq \;n^{(N)}_i.
\]
 \item[(c)]
There exists $R\in\N$ such that for any $N\in\N$ and $1\leq i\leq r$,
 \eq
 \label{finiterecassumption}
 n_i^{(N+R)}\geq \;n_i^{(N)}+1.
 \qe
  \item[(d)]
 For every $1\leq i \leq r$, there exist $q_1,\ldots,q_r\in(0,1)$ such that
 \eq
 \label{path3}
\lim_{N\rightarrow\infty}\frac{n_i^{(N)}}{N}=q_i.
\qe
 \end{enumerate}
 We then write for convenience
 \eq
 \label{conven}
 P_N(x)=P_{\bv n^{(N)}}(x),\qquad N\in\N,
 \qe
and observe that $P_N$ has degree $N$. We now focus on the weak convergence for the zero counting probability measure $\nu_N$ of $P_N$ as $N\rightarrow\infty$, defined as in \eqref{zerocountingmeasure} with $z_1,\ldots,z_N$  the zeros of  $P_N(x)$, maybe up to a rescaling of the zeros. 
Before showing how our results answer that question in the case of  the multiple Hermite and multiple Laguerre polynomials, we first need to introduce a few ingredients from  free probability theory.

\subsection{Elements of free probability}

Free probability  deals with non-commutative random variables which are  independent in an algebraic sense. It has been introduced by Voiculescu for the purpose of solving operator algebra problems. We now just provide the few elements of free probability needed for the purpose of this work, and refer to \cite{VDN,AGZ} for comprehensive introductions.

For a probability measure $\lambda$ on $\R$ with compact support, let  $K_\lambda$ be the inverse, for the composition of formal series, of the Cauchy-Stieltjes transform
\begin{align}
\label{CS}
G_\lambda(z) & =\int\frac{\lambda(\di x)}{z-x}\\
& =\sum_{k=0}^\infty \left(\,\int x^k\lambda(\di x)\right)z^{-k-1}\nonumber,
\end{align}
and set the $R$-transform of $\lambda$ by
\eq
\label{Rdef}
R_\lambda(z)=K_\lambda(z)-\frac{1}{z}.
\qe
\begin{definition} Let $\lambda$ and $\eta$ be two probability measures on $\R$ with compact support. The free additive convolution of $\lambda$ and $\eta$, denoted by  $\lambda\boxplus\eta$, is the unique probability measure (on $\R$ with compact support) which satisfies
\eq
\label{Rlin}
R_{\lambda\boxplus\eta}(z)=R_\lambda(z)+R_\eta(z).
\qe
\end{definition}

 Consider a probability measure $\lambda$ on $[0,+\infty)$  with compact support different from $\delta_0$. If   $\chi_\lambda$ is the inverse for the composition of formal series of
\eq
\frac{1}{z}G_\lambda\left(\frac{1}{z}\right)-1 = \sum_{k=1}^\infty \left(\,\int x^k\lambda(\di x)\right)z^{k},\\
\qe
 we then define the $S$-transform of $\lambda$ by
\eq
\label{Sdef}
S_\lambda(z)=\frac{1+z}{z}\chi_\lambda(z).
\qe
\begin{definition} Let $\lambda$ and $\eta$ be two probability measures on $[0,+\infty)$ with compact support and both different from $\delta_0$. The free multiplicative convolution of $\lambda$ and $\eta$, denoted  $\lambda\boxtimes\eta$, is the unique probability measure (on $[0,+\infty)$ with compact support and different from $\delta_0$) which satisfies
\eq
\label{Slin}
S_{\lambda\boxtimes\eta}(z)=S_\lambda(z)S_\eta(z).
\qe
\end{definition}

For this work, the importance of the free additive and multiplicative convolutions relies on the following results due to Voiculescu, extracted from \cite{AGZ}, which describe the limiting eigenvalue distribution of perturbed GUE and Wishart matrices. A random matrix $\bv X_N$ is distributed according to  ${\rm GUE}(N)$ if it is drawn from the space $\mathcal H_N(\C)$ of $N\times N$ Hermitian matrices according to the probability distribution 
\eq
\label{GUEdistr}
\frac{1}{Z_N}\exp\Big\{- N {\rm Tr}(\bv X_N^2)/2\Big\}\di\bv X_N,
\qe
where $\di\bv X_N$ stands for the Lebesgue measure on $\mathcal H_N(\C)\simeq \R^{N^2}$ and $Z_N$ is a normalization constant. It is said to  be distributed according to ${\rm Wishart}_{\alpha}(N)$, where $\alpha> -1$ if a real parameter, if the probability distribution reads instead
\eq
\label{Wdistr}
\frac{1}{Z_N}\det(\bv X_N)^{N\alpha}\exp\Big\{-N {\rm Tr}(\bv X_N)\Big\} \bv 1_{\big\{\bv X_N\geq \,0\big\}}\di\bv X_N,
\qe
where $\bv X_N\geq 0$ means that $\bv X_N$ is positive semi-definite. The semi-circle distribution is defined by
\eq
\label{SC}
\mu_{{\rm SC}}(\di x)=\frac{1}{2\pi}\sqrt{4-x^2}\;\bs 1_{[-2,2]}(x)\di x,
\qe
and  the (rescaled) Marchenko-Pastur distribution of parameter $\rho> 0$ by
\eq
\label{MPalpha}
\mu_{{\rm MP}(\rho)}(\di x)=\max(1-\frac{1}{\rho},0)\delta_0+\frac{1}{2\pi x}\sqrt{(\rho_+- x)(x-\rho_-)}\,\bs 1_{[\rho_-,\,\rho_+]}(x)\di x,
\qe
where $\rho_{\pm}=(1\pm\sqrt{\rho}\,)^2/\rho$.
Then the following holds.

\begin{theorem} 
\label{freepro} Consider a sequence of uniformly bounded deterministic matrices $(\bv A_N)_N$, were $\bv A_N$ is  an $N\times N$  Hermitian matrix, and assume there exists a probability measure $\lambda$ on $\R$ with compact support such that for all $\ell\in\N$,
\[
 \lim_{N\rightarrow\infty}\frac{1}{N}\Tr \big(\bv A_N\big)^\ell =\int x^\ell \lambda(\di x).
\]
\begin{itemize}
\item[{\rm (a)}]
If $(\bv X_N)_N$ is a sequence of independent random matrices with $\bv X_N$ distributed according to ${\rm GUE}(N)$, then for all $\ell\in\N$,
\[
\lim_{N\rightarrow\infty}\frac{1}{N}\,\E\Big[\Tr \big(\bv X_N+\bv A_N\big)^\ell \Big]=\int x^\ell \mu_{{\rm SC}}\boxplus\lambda(\di x).
\]
\item[{\rm (b)}]
If $(\bv X_N)_N$ is a sequence of independent random matrices with $\bv X_N$ distributed according to ${\rm Wishart}_\alpha(N)$, and if the $\bs A_N$'s are moreover positive semi-definite with $\lambda\neq \delta_0$, then for all $\ell\in\N$,
\[
\lim_{N\rightarrow\infty}\frac{1}{N}\,\E\Big[\Tr \big(\bv A_N^{1/2}\bv X_N\bv A_N^{1/2}\big)^\ell\Big] =\int x^\ell \mu_{{\rm MP}(\frac{1}{1+\alpha})}\boxtimes\lambda(\di x).
\]
\end{itemize}
\end{theorem}
 
 We are now in position to state the results of this section.

 \subsection{Multiple Hermite  polynomials}
Recall that if $H_N$ stands for the $N$-th Hermite polynomial, that is the OP associated to $\mu(\di x)=e^{-x^2/2} \di x$, then the zero counting probability distribution $\nu_N$ of its rescaled version $H_N(\sqrt{N}x)$ is known to converge weakly  towards the semi-circle distribution \eqref{SC}.

Given $r\geq 2$  pairwise distinct real numbers $a_1,\ldots,a_r$, consider the measure and the weights given by
\[
\mu(\di x)=e^{-x^2/2}\di x,\qquad w_j(x)=e^{a_jx},\qquad 1\leq j \leq r.
\]
The associated MOPs are called multiple Hermite polynomials. For a sequence of multi-indices $(\bv n^{(N)})_N$ satisfying the path-like structure described in Section \ref{MOPdef}, denote by $ H_{N}^{(a_1\,,\,\ldots\,,\,a_r)}$ the associated MOP as in \eqref{conven}. We shall prove the following.

\begin{theorem} 
\label{ThHermite}
Let $\nu_N$ be the zero probability distribution of the rescaled multiple Hermite polynomial 
\[
H_{N}^{(\sqrt{N}a_1\,,\,\ldots\,,\,\sqrt{N}a_r)}\big(\sqrt{N}x\big).
\]
Then  $\nu_N$ converges weakly as $N\rightarrow\infty$  towards 
\[\mu_{{\rm SC}}\boxplus\Big(\sum_{j=1}^rq_j\delta_{a_j}\Big).\]
\end{theorem}

Although we introduced the $R$-transform of a probability measure as a formal series, it is actually possible to define it  as a proper  analytic function, provided  one restricts oneself to appropriate subdomains of the complex plane, and  equality \eqref{Rlin} continues to hold, see \cite[Section 5]{BV}. Then, since $R_{\mu_{{\rm SC}}}(z)=z$ and the Cauchy-Stieltjes transform of $\sum_{i=1}^rq_i\delta_{a_i}$ is explicit,  one can  obtain from \eqref{Rlin} that the Cauchy-Stiejles transform $G$ of $\mu_{{\rm SC}}\boxplus\big(\sum_{j=1}^rq_j\delta_{a_j}\big)$  is an algebraic function, by performing similar manipulations than in the proof of \cite[Lemma 1]{Bi} and concluding by analytic continuation. More precisely, one obtains that 

\begin{corollary} The weak limit of $\nu_N$ has a Cauchy-Stieltjes transform $G$ which  satisfies the algebraic equation
\eq
P\big(z,G(z)\big)=0,\qquad z\in\C,
\qe
where the bivariate polynomial $P(z,w)$ is given by
\eq
\label{bivHermite}
P(z,w)=w\prod_{i=1}^r(z-w-a_i)-\sum_{i=1}^rq_i\prod_{j=1,\,j\neq i}^r(z-w-a_i).
\qe
\end{corollary}
Probability measures for which the Cauchy-Stieltjes transform is algebraic have interesting regularity properties, see \cite[Section 2.8]{AZ}, and are moreover suitable for  numerical evaluation, see e.g. \cite{ER}.

We now turn to multiple Laguerre polynomials, for which we provide a similar analysis. 
 \subsection{Multiple Laguerre  polynomials}
If  $L^{(\alpha)}_N$ stands for the $N$-th Laguerre polynomial of parameter $\alpha >-1$, that is the OP associated to $\mu(\di x)=x^\alpha e^{-x}\bs 1_{[0,+\infty)}(x)\di x$, then it is known that the zero probability distribution $\nu_N$ of $L_N^{(N\alpha)}(Nx)$ converges weakly as $N\rightarrow\infty$ towards the Marchenko-Pastur distribution  \eqref{MPalpha} of parameter $1/(1+\alpha)$.
 
There exist two different definitions for the multiple Laguerre polynomials in the literature, see \cite[Section  23.4]{Is}. We   consider here the so-called multiple Laguerre polynomials of the second kind, which are defined as follows.  Given $r\geq 2$  pairwise distinct positive numbers $a_1,\ldots,a_r$ and $\alpha\geq 0$, consider 
\[
\mu(\di x)=x^\alpha e^{-x}\bs 1_{[0,+\infty)}(x)\di x,\qquad w_j(x)=e^{(1-a_j)x}, \qquad 1\leq j \leq r,
\]
and, given  a sequence of multi-indices $(\bv n^{(N)})_N$ satisfying the path-like structure described previously, let $ L_{N}^{(\alpha\,;\,a_1\,,\,\ldots\,,\,a_r)}$  be the associated MOP as in \eqref{conven}.

\begin{theorem} 
\label{ThLaguerre}
Let $\nu_N$ be the zero probability distribution of the rescaled multiple Laguerre polynomial 
\[
L_{N}^{(N\alpha \,;\, N a_1\,,\,\ldots\,,\, N a_r)}(x).
\]
 Then  $\nu_N$ converges weakly as $N\rightarrow\infty$  towards 
\[
\mu_{{\rm MP}(\frac{1}{1+\alpha})}\boxtimes\Big(\sum_{j=1}^rq_j\delta_{1/a_j}\Big).
\]
\end{theorem}

As it was the case for the $R$-transform, the $S$-transform can  be defined as an analytic function, and \eqref{Slin} also holds on subdomains of the complex plane, see  \cite[Section 6]{BV}. Then, because $S_{\mu_{{\rm MP}	(\rho)}}(z)=\rho/(1+ \rho z)$,  one can also obtain from  \eqref{Slin}, taking care of the definition domains, that the Cauchy-Stieltjes transform $G$ of $\mu_{{\rm MP}(\frac{1}{1+\alpha})}\boxtimes\big(\sum_{j=1}^rq_j\delta_{1/a_j}\big)$ satisfies an algebraic equation.

\begin{corollary} The weak limit of $\nu_N$ has a Cauchy-Stieltjes transform $G$ which satisfies the algebraic equation
\eq
P\big(z,G(z)\big)=0,\qquad z\in\C,
\qe
where $P(z,w)$ is given by
\eq
\label{bivLaguerre}
P(z,w)=w\prod_{i=1}^r\Big(z-\frac{zw}{a_i}+\frac{\alpha }{a_i}\,\Big)-\sum_{i=1}^rq_i\prod_{j=1,\,j\neq i}^r\Big(z-\frac{zw}{a_i}+\frac{\alpha }{a_i}\,\Big).
\qe
\end{corollary}

\subsection{Proofs} 
\label{proofs4} 

Before providing proofs for Theorems \ref{ThHermite} and \ref{ThLaguerre}, we first  precise a few points  concerning MOP Ensembles, that we introduced in Section \ref{MOP1}. 

A sequence of measures $(\mu_N)_N$,  weights $w_{j,N}\in L^2(\mu_N)$, $1\leq j\leq r$, and a path-like sequence of multi-indices $(\bv n^{(N)})_N$   induce a sequence of MOP Ensembles. Namely, for each $N$ one can associate random variables $x_1,\ldots,x_N$  distributed according to \eqref{MOPdensity} where we chose for the multi-index $\bv n=\bv n^{(N)}$. As Biorthogonal Ensembles, one can chose  $P_{k,N}$ to be the $\bv n^{(k)}$-th (type II) MOP associated with $\mu_N$ and the $w_{j,N}$'s. The associated biorthogonal functions $Q_{k,N}$'s can then be constructed for any $k\in\N$ from the type I MOPs, see \cite[Theorem 23.1.6]{Is}, and Assumption \ref{structassump} is satisfied with $\frak q_N=1$. We moreover recall that the average characteristic polynomial $\chi_N$ equals $P_{N,N}$.

In order to obtain growth estimates for the  $\langle xP_{k,N},Q_{m,N}\rangle_{L^2(\mu_N)}$'s, we now describe a connection with the so-called nearest neighbors recurrence coefficients, which are sometimes easier to compute.

\subsubsection{Nearest neighbors recurrence coefficients}

Van Assche \cite{VA2}  established for general MOPs, says associated to a measure $\mu$ and weights $w_i$'s, that for every normal multi-index $\bv n$ there exist  real numbers $(a_{\bv n}^{(d)})_{1\leq d \leq r}$ and $(b_{\bv n}^{(d)})_{1\leq d \leq r}$ satisfying 
\begin{align}
\label{nearrec}
xP_{\bv n}(x) & = P_{\bv n + \bv e_1} + a_{\bv n}^{(1)} P_{\bv n}(x) + \sum_{d=1}^rb_{\bv n}^{(d)}P_{\bv n  - \bv e_d}(x),\nonumber\\
\vdots &  \\
xP_{\bv n}(x) & = P_{\bv n + \bv e_r} + a_{\bv n}^{(r)} P_{\bv n}(x) + \sum_{d=1}^rb_{\bv n}^{(d)}P_{\bv n  - \bv e_d}(x),\nonumber
\end{align}
where 
\[
\bv e_d=(\,\underbrace{0,\ldots,0}_{d-1},1,0,\ldots,0)\in\N^r,\qquad 1\leq d \leq r.
\]
Note that this provides 
\eq
\label{Lrelation}
P_{\bv n + \bv e_i}(x)- P_{\bv n + \bv e_j}(x)=(a_{\bv n}^{(j)}-a_{\bv n}^{(i)})P_{\bv n}(x),\qquad 1\leq i,j\leq r.
\qe

With the path-like sequence of multi-indices $(\bv n^{(k)})_{k\in\N}$ and allowing the $w_i$'s and $\mu$ to depend on a parameter $N$, we write  for convenience 
\[
a_{k,N}^{(d)}=a_{\bv n^{(k)},\,N}^{(d)}\;, \qquad b_{k,N}^{(d)}=b_{\bv n^{(k)},\,N}^{(d)},\qquad 1\leq d \leq r.
\]
Then the following holds.
\begin{lemma}
\label{criterion2}
If there exists $\varepsilon>0$ such that for every $1\leq d \leq r$ the sequences 
\eq
\label{cond1}
\Bigg\{\max_{k\in\N\,:\;\left|\frac{k}{N}-1\right|\,\leq\, \varepsilon}\;\max_{j=1}^r\big|a_{\bv n^{(k)}-\bv e_j,N}^{(d)}\big|\Bigg\}_{N\geq 1}, \qquad 
\Bigg\{\max_{k\in\N\,:\;|\frac{k}{N}-1|\,\leq\, \varepsilon}\big|b_{k,N}^{(d)}\big|\Bigg\}_{N\geq 1},
\qe
are bounded, then  so is the sequence
\[
\label{sequence2}
\left\{\;\max_{k,m\in\N \,:\;\left|\frac{k}{N}-1\right|\leq \,\varepsilon,\;\left|\frac{m}{N}-1\right|\leq \,\varepsilon}\left|\langle xP_{k,N},Q_{m,N}\rangle_{L^2(\mu_N)}\right|\;\right\}_{N\geq 1}.
\]
\end{lemma}

\begin{proof}
First, as a consequence of \eqref{finiterecassumption} and \cite[(23.1.7)]{Is}, we have 
\eq
\label{zeroreccoef}
\langle xP_{k,N},Q_{m,N}\rangle_{L^2(\mu_N)}=0,\qquad  m<k-R.
\qe
Define the sequence  $(i_k)_{k\in\N}$ taking its values in $\{1,\ldots,r\}$ by
\[
\bv n^{(k+1)}=\bv n^{(k)}+\bv e_{i_{k}},Ê\qquad m\in\N.
\] 
For a fixed $k$, which may be chosen as large as we want, \eqref{nearrec}   yields
\eq
\label{rec1}
xP_{k,N}(x)=P_{k+1,N}(x)+a_{k,N}^{(i_{k})}P_{k,N}(x)+\sum_{d=1}^r b_{k,N}^{(d)}P_{\bv n^{(k)}-\bv e_d,N}(x).
\qe
Then, since \eqref{Lrelation} provides  for any  $1\leq d \leq r$ and $m$  large enough
\[
P_{\bv n^{(m)} - \bv e_d,N}(x)=P_{m-1,N}(x)+\big(a_{\bv n^{(m-1)}-\bv e_d,N}^{(d)}-a_{\bv n^{(m-1)}-\bv e_d,N}^{(i_{m-1})}\big)P_{\bv n^{(m-1)} -\bv e_d,N}(x),\\
\]
we obtain inductively  with  \eqref{rec1}  that 
\begin{multline}
\label{rec2}
xP_{k,N}(x) =P_{k+1,N}(x)+a_{k,N}^{(i_{k})}P_{k,N}(x)+\Big(\sum_{d=1}^r b_{k,N}^{(d)}\Big)P_{k-1,N}(x)\\
   + \sum_{m=k-R}^{k-2}\left(\;\sum_{d=1}^r b_{k,N}^{(d)}\prod_{l=m+1}^{k-1}\big(a_{\bv n^{(l)}-\bv e_d,N}^{(d)}-a_{\bv n^{(l)}-\bv e_d,N}^{(i_{l})}\big)\right)P_{m,N}(x) + R_{k,N}(x),
\end{multline}
where $R_{k,N}$ is a polynomial of degree at most $k-R-1$. By comparing \eqref{rec2} with the  (unique) decomposition \eqref{decomposition} and \eqref{zeroreccoef}, we  obtain  explicit formulas for the $\langle xP_{k,N},Q_{m,N}\rangle$'s  in terms of the nearest neighbor recurrence coefficients, from which Lemma \ref{criterion2} easily follows.
\end{proof}

\subsubsection{Proof of Theorem \ref{ThHermite}}

\begin{proof}
Associate to the multi-indices $(\bv n^{(N)})_{N\in\N}$ the (uniformly bounded) sequence $(\bv A_N)_{N\in\N}$ of diagonal matrices
\[
\bv A_N ={\rm diag}\big(\, \underbrace{a_1,\ldots, a_1}_{n_1^{(N)}}, \,\ldots\, , \underbrace{a_r,\ldots, a_r}_{n_r^{(N)}} \,\big)\in\mathcal H_N(\C).
\] 
On the one hand, let $(\bv X_N)_N$ be a sequence of independent random matrices, with $\bv X_N$ distributed according to ${\rm GUE}(N)$. If  $\hat\mu^N$ stands for the empirical measure associated to the eigenvalues of  $\bv Y_N=\bv X_N+\bv A_N$, then Theorem \ref{freepro} (a) and \eqref{path3} provide   for any $\ell\in\N$ 
\begin{align}
\label{free+}
\lim_{N\rightarrow\infty}\mathbb E\left[\,\int x^\ell \hat\mu^N(\di x)\right] & = \lim_{N\rightarrow\infty}\frac{1}{N}\,\E\Big[\Tr \big(\bv X_N+\bv A_N\big)^\ell \Big]\nonumber\\
 & =\int x^\ell \mu_{{\rm SC}}\boxplus\Big(\sum_{j=1}^rq_j\delta_{a_j}\Big)(\di x).
\end{align}
On the other hand, observe from \eqref{GUEdistr} that the random matrix $\bv Y_N$ is distributed on $\mathcal H_N(\C)$ according to
\eq
\label{GUEextsource}
\frac{1}{Z'_N}\exp\Big\{-N{\rm Tr}\big(\bv Y_N^2-2\bv A_N\bv Y_N\big)/2\Big\}\di \bv Y_N,
\qe
where $Z_N'$ is a new normalization constant. By performing a spectral decomposition in \eqref{GUEextsource}, integrating out the eigenvectors and using a confluent version of the Harish-Chandra-Itzykson-Zuber formula, Bleher and Kuijlaars \cite{BK} obtained   that the random eigenvalues of $\bv Y_N$ form a MOP Ensemble, see \eqref{MOPdensity}, associated to  the $N$-dependent  weights and  measure 
\eq
\label{Herrescmeas}
\mu_N(\di x)=e^{-Nx^2/2}\di x,\qquad w_{j,N}(x)=e^{Na_jx},\qquad 1\leq j \leq r,
\qe
 and the multi-index $\bv n^{(N)}$. The average characteristic polynomial $\chi_N$ for that MOP Ensemble then equals the associated $\bv n^{(N)}$-th MOP, which is  seen from a change of variable to be  $H_{N}^{(\sqrt{N}a_1\,,\,\ldots\,,\,\sqrt{N}a_r)}\big(\sqrt{N}x\big)$, up to a multiplicative constant. The weights in  \eqref{Herrescmeas} form an AT system, from which it follows that any multi-index is normal, and that $\chi_N$ has real zeros, cf.  \cite[Chapter 23]{Is}. One moreover obtains from \cite[Section 5.2]{VA2} and a change of variables explicit formulas for the nearest neighbors recurrence coefficients associated to \eqref{Herrescmeas}, 
\[
a^{(d)}_{\bv n,N}=a_d,\qquad b^{(d)}_{\bv n,N}=\frac{n_d}{N},\qquad \bv n =(n_1,\ldots,n_r).
\] 
Thus, Theorem \ref{ThHermite}  follows from \eqref{free+}, Lemma \ref{criterion2} and Corollary \ref{part->zero}.
\end{proof}

\subsubsection{Proof of Theorem \ref{ThLaguerre}}
\begin{proof}
The proof follows the same spirit as  the proof of Theorem \ref{ThHermite}. Introduce the sequence of  (uniformly bounded) diagonal matrices 
\[
\bv A_N ={\rm diag}\big(\, \underbrace{1/a_1,\ldots,1/a_1}_{n_1^{(N)}}, \,\ldots\, , \underbrace{1/a_r,\ldots, 1/a_r}_{n_r^{(N)}} \,\big),
\] 
and let $(\bv X_N)_N$ be a sequence of independent   random matrices, where $\bv X_N$ is distributed according to  ${\rm Wishart}_\alpha(N)$.
With $\hat\mu^N$ the empirical measure of the eigenvalues of $\bv Y_N=\bv A_N^{1/2}\bv X_N\bv A_N^{1/2}$, Theorem \ref{freepro} (b) and \eqref{path3} then provide for all $\ell\in\N$
\eq
\label{freex}
\lim_{N\rightarrow\infty}\mathbb E\left[\,\int x^\ell \hat\mu^N(\di x)\right]=\int x^\ell \mu_{{\rm MP}(\frac{1}{1+\alpha})}\boxtimes\Big(\sum_{j=1}^rq_j\delta_{1/a_j}\Big)(\di x).
\qe
Now, observe from \eqref{Wdistr} that $\bv Y_N$ is distributed on $\mathcal H_N(\C)$ according to
\eq
\label{Wishextsource}
\frac{1}{Z'_N}\det(\bv Y_N)^{N\alpha}\exp\Big\{- N{\rm Tr}\big(\bv A_N^{-1}\bv Y_N\big)\Big\}\bv 1_{\big\{\bv Y_N\geq \,0\big\}}\di \bv Y_N,
\qe
where $Z_N'$ is a new normalization constant. Similarly than for the Hermite case, the eigenvalues of $\bv Y_N$ form a MOP Ensemble associated to 
\eq
\label{Lagrescmeas}
\mu_N(\di x)=x^{N\alpha}e^{-Nx}\di x,\qquad w_{j,N}(x)=e^{N(1-a_j)x},\qquad 1\leq j \leq r,
\qe
 and the multi-index $\bv n^{(N)}$, see \cite{BK2}. The average characteristic polynomial $\chi_N$ is then the $\bv n^{(N)}$-th MOP associated to \eqref{Lagrescmeas}, which is $L_{N}^{(N\alpha \, ; \, Na_1\,,\,\ldots\,,\,Na_r)}\big(x\big)$ up to a multiplicative constant. The weights in  \eqref{Lagrescmeas} form an AT system so that any multi-index is normal and $\chi_N$ has real zeros. If we denotes $|\bv n|=n_1+\cdots+ n_r$ for $\bv n\in\N^r$, then one obtains from \cite[Section 5.4]{VA2} that the nearest neighbors recurrence coefficients for \eqref{Herrescmeas} read
\[
a^{(d)}_{\bv n,N}=\frac{n_d(|\bv n|+N\alpha )}{N^2a_d},\qquad b^{(d)}_{\bv n,N}=\frac{|\bv n|+N\alpha+1 }{Na_d}+\sum_{j=1}^r\frac{n_j}{Na_j},\qquad \bv n =(n_1,\ldots,n_r).
\] 
Theorem \ref{ThLaguerre}  finally  follows from \eqref{freex}, Lemma \ref{criterion2} and Corollary \ref{part->zero}.
\end{proof}

\begin{remark} Having in mind the proofs of Theorems \ref{ThHermite} and \ref{ThLaguerre}, it would be of interest to find out if there exists a matrix model for the multiple version of the Jacobi polynomials, the Jacobi-Pi\~ neiro polynomials, which are related in a limiting case to the rational approximations of $\zeta(k)$ and polylogarithms  \cite[Section 23.3.2]{Is}, and then if it would be possible to describe its limiting zero distribution thanks to free convolutions. 
\end{remark}

\paragraph*{Acknowledgements} The author would like to thank Walter Van Assche for the references \cite{Nev,VA},  Steven Delvaux for pointing out the relation \eqref{Lrelation} and for the interesting discussion which followed, Franck Wielonsky for noticing a few typos concerning multiple Laguerre polynomials in an earlier version of this work, and the anonymous referee for improving the readability of the paper. 

The author is supported by FWO-Flanders projects G.0427.09 and by the Belgian Interuniversity
Attraction Pole P07/18.

\end{document}